\documentclass[12pt]{amsart}
 
\usepackage {amsfonts, amssymb, amscd, amsthm, amsmath}
\usepackage{geometry}
\usepackage[bookmarksopen, naturalnames]{hyperref}
\usepackage{tikz}
\newtheorem{Thm}{Theorem}[section]
\newtheorem{theorem}[Thm]{Theorem}
\newtheorem*{theorem*}{Theorem}

\newtheorem{proposition}[Thm]{Proposition}

\newtheorem{remark}[Thm]{Remark}

\newtheorem{example}[Thm]{Example}

\renewcommand{\phi}{\varphi}
\newcommand{\C}{\mathbb{C}}
\newcommand{\co}{\mathbb{C}}

\newcommand{\D}{\mathbb{D}}
\newcommand{\Dh}{\widehat{\mathbb{D}}}
\newcommand{\T}{\mathbb{T}}

\title{New classes of hypercyclic Toeplitz operators}

\author[E. Abakumov, A. Baranov, S. Charpentier, and A. Lishanskii]{Evgeny Abakumov, Anton Baranov,
St\'ephane Charpentier,\\ and Andrei Lishanskii }

\address{Evgeny Abakumov\\
	 LAMA, Univ Gustave Eiffel, Univ Paris Est Creteil, CNRS, F-77454 Marne-la-Vall\'ee, France}
\email{evgueni.abakoumov@univ-eiffel.fr}

\address{Anton Baranov\\
	Department of Mathematics and Mechanics, St. Petersburg State University, 28, Universitetskii prosp., St. Petersburg, 1984504, Russia}
\email{anton.d.baranov@gmail.com}

\address{St\'ephane Charpentier\\
	Aix-Marseille Universit\'e,
	Centre de Math\'ematiques et Informatique,
	39, rue F. Joliot Curie, 13453 Marseille Cedex 13, France}
\email{stephane.charpentier.1@univ-amu.fr}

\address{Andrei Lishanskii\\
	Department of Mathematics and Computer Science, St. Petersburg State University, 
29B, 14 line of the Vasilievsky Island, St. Petersburg,  199178, Russia}
\email{lishanskiyaa@gmail.com}

\thanks{The results of Sections 2 and 3 were obtained with the support of Ministry of Science and Higher Education 
of the Russian Federation, agreement No 075-15-2019-1619.  
The results of Section 4 were obtained with the support of Russian Science Foundation grant  19-11-00058.}



\begin{document}

\begin{abstract}
We study hypercyclicity of  Toeplitz operators in the Hardy space 
$H^2(\mathbb{D})$ 
with symbols of the form $R(\overline{z}) +\phi(z)$, where $R$ is a rational 
function and $\phi \in H^\infty(\mathbb{D})$. 
We relate this problem to cyclicity of certain families of functions
for analytic Toeplitz operators
and give new sufficient conditions for hypercyclicity
based on deep results of B. Solomyak. 
\end{abstract}

\maketitle
\sloppy

\section{Introduction and main results}

Toeplitz operators are among the most important objects in operator theory.
They have numerous applications in complex analysis, theory of orthogonal polynomials, 
mathematical physics, etc.
Recall that for a function $\psi \in L^\infty(\mathbb{T})$ the {\it Toeplitz operator}
$T_\psi: H^2 \to H^2$ with the symbol $\psi$  is defined as 
$T_\psi f = P_+(\psi f)$, where $P_+$ stands for the orthogonal projection from
$L^2(\mathbb{T})$ onto  $H^2$.
As usual, $\mathbb{D}$ and $\mathbb{T}$ denote the unit disc and 
the unit circle, respectively. We recall that the Hardy space $H^2$ of the unit disc $\D$
consists of functions analytic in $\D$  whose Taylor coefficients are square summable,
and   $H^\infty$ is the space of all bounded analytic functions in $\D$.

In the present paper we study the hypercyclicity property for a class of Toeplitz 
operators on $H^2$. A continuous linear operator $T$ on a separable Banach space 
$X$  is said to be  \textit{hypercyclic} if 
there exists $x \in X$ such that the set 
$\{T^n x:\, n\in\mathbb{N}_0\}$ is dense in $X$
(here $\mathbb{N}_0 = \{0,1,2, \dots\}$). 
Toeplitz operators with antianalytic symbols were among the
basic examples of hypercyclic operators. In 1969, 
S. Rolewicz \cite{rol} showed that the operator $T_{\alpha\overline{z}}$ 
(a multiple of the backward shift) is hypercyclic on $H^2$ 
whenever $|\alpha|>1$. Later, G. Godefroy and J. Shapiro \cite{gosh} 
showed that for a function 
$\phi \in H^\infty$ the antianalytic Toeplitz operator $T_{\overline{\phi}}$
is hypercyclic if and only if $\phi(\mathbb{D}) \cap \mathbb{T} \ne \emptyset$.
On the other hand, it is obvious that there are no 
hypercyclic Toeplitz operators with analytic symbols
(i.e., among multiplication operators). 

However, the problem of describing hypercyclic Toeplitz operators 
in the general case seems to be largely open. 
This problem was explicitly stated by S. Shkarin \cite{sh} who described
hypercyclic Toeplitz operators with symbols of the form $\Phi(z) =a\overline{z} +b +cz$
(i.e., with tridiagonal matrix). A. Baranov and A. Lishanskii 
\cite{barl} gave necessary and (separately) sufficient conditions for hypercyclicity 
of Toeplitz operators with polynomial antianalytic part of the symbol,
that is, for symbols $\Phi$ of the form $\Phi(z) = P\big(\frac{1}{z}\big) +\phi(z)$,
where $P$ is a polynomial and $\phi\in H^\infty$.
It turned out that the valence of the symbol played the central role in this study.
\medskip{}

In the present paper we will consider more general symbols. Let $R$
be a rational function without poles in $\overline{\mathbb{D}}$, i.e., 
\begin{equation}
\label{r}
R(z) = P(z) +\sum _{l=1}^r \sum_{j=1}^{k_l}\frac{\alpha_{l, j}}{(z-\eta_l)^{j}},
\end{equation}
where $P(z)=\sum_{k=0}^{N_1} c_k z^k$  is a polynomial of degree $N_1$, 
$\alpha_{l, j} \in \mathbb{C}$, 
and $\eta_l \in \widehat{\mathbb{D}} := \C \setminus \overline{\D}$, $1\leq l\leq r$, 
are the (distinct) poles of $R$ of multiplicities $k_l$. 
We put $N_2 =\sum_{l=1}^r  k_l$ and denote by $N=N_1+N_2$ the degree of the rational function $R$.
We do not exclude the cases when $N_1=0$ (i.e., $P=const$) or $N_2=0$ ($R=P$).
Finally, let $\varphi\in H^{\infty}$, and put
\begin{equation}
\label{fg}
\Phi(z)= R\bigg(\frac{1}{z}\bigg)+\varphi(z).
\end{equation}
Thus, $\Phi$ is analytic in $\D$ except the poles $1/\eta_1, \dots,1/\eta_r$.
The aim of this note is to describe new classes of hypercyclic Toeplitz operators $T_\Phi$
with the symbols of the form \eqref{fg}. 
\medskip

As in \cite{barl}, our conditions will be formulated in terms of the valence of $\Phi$. 
Recall that an analytic function $h$ in a domain
$D$ is said to be {\it $n$-valent in} $D$ if the equation $h(z) = w$ has 
{\it at most} $n$ solutions in $D$ counting multiplicities for any $w\in \C$. It is said to be 
\emph{exactly $n$-valent in} $D$ if this equation has exactly $n$ solutions (counting multiplicities) 
for any $w\in h(\D)$. Note that $\Phi$ (given by \eqref{fg}) has a pole of order $N_1$ at zero 
and poles of multiplicities $k_l$ at $1/\eta_l$, whence the equation $\Phi(z) = w$ has exactly 
$N$ solutions when $|w|$ is sufficiently large.

Let us first illustrate the connection between the valence of $\Phi$ and the hypercyclicity of $T_{\Phi}$ through a necessary condition.
For the case of symbols with polynomial antianalytic part the statement was proved in \cite{barl}.

\begin{proposition}
\label{ness}
Let $\Phi$ be given by \eqref{fg}. If $T_{\Phi}$ is hypercyclic, then 
the function $\Phi$ is $N$-valent in $\D$. 
\end{proposition}

Let us also note that if $T_\Phi$ is hypercyclic, then we have
$$
\sigma(T_\Phi) \cap \T \ne \emptyset, \qquad 
\sigma(T_\Phi)\cap \Dh \ne \emptyset.
$$
The first condition is a basic property that  holds for any 
hypercyclic operator \cite[Theorem 5.6]{gp}, while the second one follows from the fact
that, for any Toeplitz operator, its spectral radius coincides with its norm \cite[Part B, Ch. 4]{nik}. 

It can be proven that if $\Phi$ is $N$-valent, then
$\sigma(T_\Phi) = \mathbb{C} \setminus \Phi(\D, N)$, where $\Phi(\D, N)$ 
consists of those $w\in \C$ for which $\Phi(z) = w$ has exactly $N$ solutions in 
$\D$. We refer to \cite[Proposition 2.2]{barl} for the details.
We do not know whether 
the hypercyclicity of $T_\Phi$ implies that $\sigma(T_\Phi) \cap \D \ne \emptyset$.
\medskip

The key step of the proof of our sufficient conditions 
for hypercyclicity is an application of the following principle, already  used in \cite{barl}:
Hypercyclicity of $T_\Phi$ can be reduced to standard cyclicity of some analytic
Toeplitz (i.e., multiplication) operator. We recall that a finite (or countable) family $U$ of functions in $H^2$
is said to be {\it cyclic} for $T_h$, $h\in H^{\infty}(\D)$, if the family $\{h^k u: u\in U, k\ge 0\}$ 
is complete in $H^2$. More specifically,
for $\lambda \in \C \setminus \overline{\Phi(\D)}$, let us set
\begin{equation}
\label{hhh}
h_{\lambda}(z) = \frac{1}{\Phi(z) - \lambda}.
\end{equation} 
Clearly, $h_{\lambda}$ is analytic and bounded in $\D$.

\begin{proposition}
\label{red}
Let $\Phi$ be given by \eqref{fg}. We assume that 
\begin{equation}
\label{spe}
\D \cap (\co \setminus \Phi(\D)) \ne \emptyset\quad{\rm and} \quad
\Dh \cap (\co \setminus \Phi(\D)) \ne \emptyset,
\end{equation}
and that, for any $\lambda \in \C \setminus \overline{\Phi(\D)}$, 
the family $\{1,z,\ldots,z^{N-1} \}$ is cyclic for $T_{h_\lambda}$ with $h_{\lambda}$ given by \eqref{hhh}.
Then $T_{\Phi}$ is hypercyclic. 
\end{proposition}
\medskip

Based on this important observation, we will then formulate several sufficient conditions 
for the hypercyclicity of $T_{\Phi}$. 
From now on, we assume that $\Phi$ is given by \eqref{fg} and satisfies \eqref{spe}. 
We denote by $A(\D)$ the disc algebra, i.e., the space of functions 
continuous in $\overline{\D}$ and analytic in $\D$.

The following result concerning the case of constant and maximal valence of the symbol 
is a direct extension of \cite[Theorem 1.2]{barl}. 

\begin{theorem}
\label{main2}
Assume that $\phi \in A(\D)$ and $\Phi$ satisfies the following Maximal Valence Condition\textup:
\medskip

{\bf (MVC)} the function $\Phi$ is exactly $N$-valent on $\overline{\D}\setminus 
\{0,1/\eta _1,\ldots,1/\eta_r\}$,
i.e., for any $w \in \Phi(\overline{\D})$ the equation $\Phi(z) = w$ has exactly
$N$ solutions in $\overline{\D}$ counting \medskip
multiplicities.

Then $T_{\Phi}$ is hypercyclic.
\end{theorem}

Theorem \ref{main2}, however, does not apply if $\Phi$ has varying valence in $\overline{\D}$.
A novel feature of the present note is that we make use of deep results of B. Solomyak about cyclic families
for analytic Toeplitz operators. Together with Proposition \ref{red}, this allows us to significantly enlarge 
the class of symbols inducing hypercyclic Toeplitz operators and, 
in particular, to include symbols with varying valence. These results are new even for 
symbols with polynomial antianalytic part studied in \cite{barl}.

Cyclicity of families of functions for analytic Toeplitz operator is very far from being understood; deep results 
in this direction were obtained by  B.~Solomyak \cite{sol} and B.~Solomyak and A.~Volberg \cite{solv} in 1980-s.
We use the geometric conditions on the valence of the function $\Phi$ given in 
\cite{sol} to distinguish two more classes of hypercyclic Toeplitz
operators.

\begin{theorem}
\label{main3}
Assume that $\phi\in A(\D)$ and, for some $\lambda \in \C \setminus \overline{\Phi(\D)}$,
the function $h_{\lambda}$ given by \eqref{hhh} satisfies the 
Increasing Argument Condition\textup:
\medskip

{\bf (IAC)} 
The set $h_{\lambda}(\T)$ is a finite union of $C^2$-smooth Jordan arcs and
some continuous branch of the function $t \mapsto \arg h_{\lambda}(e^{it})$ is \textup(strictly\textup) 
increasing on $[0, 2\pi]$.
\medskip

Then $T_{\Phi}$ is hypercyclic.
\end{theorem}
\medskip

To state the last (and, in a sense, more general) result, we need to introduce some 
notations (see \cite[Paragraph 1.3]{sol}).
Let $h$ be a meromorphic function in $\overline{\D}$ (i.e., meromorphic in $\{|z|<\rho\}$ for some
$\rho>1$) with poles $\eta_1,\ldots,\eta_s$ in $\D$. Then the set $h(\T)$ 
splits the plane into finitely many connected components. 
Let $\Omega_h^{(k)}$ be the union of the connected components of $\mathbb{C}\setminus h(\T)$ 
in which the number of pre-images in $\D\setminus \{\eta_1,\ldots,\eta_s\}$ of $h$ is equal to $k$ 
(counting multiplicity). The bounded connected components of the set 
$\co \setminus \overline{h(\D}) = :\Omega_h^{(0)}$ will be called \textit{holes}.

We say that $h$ is a function {\it of general position} if it is analytic in a neighborhood of $\T$,
the curve $t \mapsto h (e^{it})$ has only finitely many points of self-intersection which 
are simple and transversal, and $h'|_{\T} \neq 0$.

Finally, let us say that two disjoint open sets $U$ and $V$ in $\mathbb{C}\setminus h(\T)$ are 
\emph{adjacent} if $\overline{U}\cap \overline{V}$ contains $h(I)$ for some open arc $I \subset \T$. 
Then we remark that if $h$ is meromorphic in $\overline{\D}$ and of general position, then 
a connected component of $\Omega_h^{(k_1)}$ and a connected component of $\Omega_h^{(k_2)}$ 
are adjacent if and only if $k_1=k_2 -1$ or $k_1=k_2 +1$. 
Moreover, if for some $\lambda \in \Omega_h^{0}$, $h_{\lambda}$ is 
given by \eqref{hhh}, then $h_{\lambda}$ is of general position if and only if so $\Phi$ is. 

\begin{theorem}
\label{main4}
Assume that $\Phi$ is of general position and satisfies the following 
Decreasing Valence Condition\textup:
\medskip

{\bf (DVC)}
For some $\lambda_0 \in \D\cap \Omega_{\Phi}^{(0)}$ and 
$\lambda_1 \in \widehat{\D} \cap \Omega_{\Phi}^{(0)}$, for 
any $j\in \{0,1\}$, and for any component $G \subset \Omega_{\Phi}^{(k)}, k \geq 1,$ 
there exist connected components $G_i$ of $\Omega_{\Phi}^{(i)}$, $1\leq i \leq k$, 
such that $G=G_k$, the components  $G_i$ and $G_{i-1}$ are adjacent for $1 < i \le k$, and $G_1$ is adjacent to the 
hole which contains $\lambda _j$.
\medskip

Then $T_{\Phi}$ is hypercyclic.
\end{theorem}

We shall observe that under the assumption of the previous theorem, 
condition (DVC) also reads as follows: For some $\lambda_0 \in \D\cap \Omega_{\Phi}^{(0)}$ 
and some $\lambda_1 \in \widehat{\D} \cap \Omega_{\Phi}^{(0)}$, and for any 
$w\in \C \setminus \overline{\Phi(\D)}$, there exist continuous paths 
$\gamma_0,\,\gamma_1:[0,1]\rightarrow \C$ with $\gamma_i(0)=w$ and 
$\gamma_i(1)=\lambda_i$ such that $\gamma_i \cap \Omega_{\Phi}^{(k)}$ is connected for any $k \geq 0$ and $i=0,1$.
\medskip{}

\begin{remark}
\label{DVCprime}
{\rm Let $\lambda_0 \in \D\cap \Omega_{\Phi}^{(0)}$ and 
$\lambda_1 \in \widehat{\D} \cap \Omega_{\Phi}^{(0)}$. Then it is easily seen that $\Phi$ satisfies 
(DVC) for $\lambda _0$ and $\lambda _1$ if and only if both functions $h=\frac{1}{\Phi-\lambda_0}$ 
and $h=\frac{1}{\Phi-\lambda_1}$ are analytic in $\overline{\D}$, are of general position, 
and satisfy the following 
\medskip
condition:

\textbf{(DVC')}} For any component $G \subset \Omega_{h}^{(k)}, k \geq 1,$ there 
exist connected components $G_i$ of $\Omega_{h}^{(i)}$, $1\leq i \leq k$, 
such that $G=G_k$, the components $G_i$ and $G_{i-1}$ are adjacent for $1 < i \le k$, and $G_1$ is adjacent to the 
unbounded connected component of $\Omega_{h}^{(0)}$. 
\end{remark}

Pictures of domains having this property or not can be found in \cite[Fig. 1, Page 812]{sol}
and in Figures \ref{fig:fig1} and \ref{fig:fig2} below.
\medskip{}

In the next section we provide examples 
of symbols $\Phi$ of the form \eqref{fg} satisfying (MVC), (IAC) or (DVC).
The proofs of Propositions \ref{ness}--\ref{red} and Theorems \ref{main2}--\ref{main4} 
are postponed to Sections \ref{nec} and \ref{suff}.
\bigskip


\section{Examples}

\begin{example}
\label{ex1}
{\rm Let $\Psi$ be a univalent function in $A(\D)$ which maps
$\D$ onto a Jordan domain $\Omega$ with $0\in \Omega$ and let $B$ be a finite Blaschke product
of degree $N$. Let $\gamma\in \D$ be such that $\Psi(\gamma) = 0$. Then 
$$
\Phi(z) = \frac{1}{\Psi\circ B(z)}  = \sum_{l=1}^r \sum_{j=1}^{k_l} 
\frac{a_{l, j}}{(z-\lambda_l)^j} +\phi(z),
$$
where $\lambda_1, \dots \lambda_l$ are such that $B(\lambda_l) = \gamma$ with multiplicity $k_l$
and $\phi\in A(\D)$. Thus, $\Phi$ is a symbol of the form \eqref{fg} which satisfies (MVC). }
\end{example}

\begin{example}
\label{ex2}
{\rm Another example of a symbol satisfying (MVC) can be obtained as follows. 
Let $\Omega$ be a Jordan domain invariant under rotation by $2\pi /N$, for some $N \in \mathbb{N}$. 
and let $\Psi: \D \to \Omega$ be the conformal mapping with $\Psi(0)=0$.
Then it is clear that $\Phi = 1/\Psi^N$ will be exactly $N$-valent in $\D$.}
\end{example}

\begin{example}
\label{ex3}
{\rm It is easy to provide with examples of $\Phi$ satisfying (IAC) as small 
perturbation of $z^n$. Indeed, let $\varepsilon>0$, $\lambda \in \C$ and 
$\Psi \in A(\D)\cap\mathcal{C}^1(\overline{\D})$. Then one defines 
$$
\Phi(z)=\lambda + [z^n(1+\varepsilon\Psi(z))]^{-1}.
$$
Clearly, $\Phi$ is of the form $P(1/z)+\varphi(z)$, and a direct computation gives 
that the derivative of $t\mapsto \arg h (e^{it})$ is equal to $n+ \mathcal{O}(\varepsilon)$, 
where $h= (\Phi -\lambda)^{-1}$. So $\Phi$ satisfies (IAC) for any $\varepsilon$ small enough. }
\end{example}
             
\begin{example}
\label{ex4}
{\rm Let us first build a $2$-valent function $h\in H(\overline{\D})$, of general position, which satisfies 
(DVC') and has a zero of order $2$ at $0$. To do so, let us first (the steps are illustrated in Figure 
\ref{fig:fig1}) consider the conformal map $\Psi$ which sends $\D$ onto the sector of annulus $\Gamma$ given by
\[
\Gamma=\left\{z\in \C:\, r<|z|<R,\,-\frac{\pi}{4}+\alpha < \arg z < \pi-\alpha\right\},
\]
with $0<r<R<(2\sqrt{2}-1)r$, $0< \alpha < \frac{\pi}{8}$. Then step 2 consists in composing $\Psi$ with  
the map $(1+\varepsilon z)z^2-\beta (1+i)$, where $\varepsilon >0$ is very small and 
$\beta$ is such that $0<\tilde c< c< \tilde a<a< \tilde b<b<\tilde d <d$ (see Figure 
\ref{fig:fig1} for the definitions of $a,\tilde a,b, \tilde b,c,\tilde c,d,\tilde d$).
Note that such $\beta$ exists because of the choice of $r$ and $R$. 
Thus, the resulting map is $g(z) = (1+\varepsilon \Psi(z))\Psi(z)^2 - \beta (1+i)$.
There is, however, a small problem that $g$ is not analytic in $\overline{\D}$. To overcome it
consider $g(\rho z)$ for some $\rho \in (0,1)$ close enough to $1$. Finally, put
$$
h(z)=g^2(\rho z) = \big[(1+\varepsilon \Psi(\rho z))\Psi(\rho z)^2 - \beta (1+i)\big]^2. 
$$
The resulting domain $h(\D)$ is shown in the last part of Figure \ref{fig:fig1}. 

It is easily seen that $h$ is $2$-valent, of general position, has zero of order $2$ at $0$ 
and satisfies (DVC').
Note also that $(h-\mu)^{-1}$ also satisfies (DVC') for any $\mu \notin h(\overline{\D})$.

If we put $\Phi = 1/h$, then it is clear that $\Phi (z) = a_1 z^{-2} +a_2 z^{-1} + \phi(z)$,
where $\phi$ is analytic in $\D$. Thus, we can take $\lambda_0\in \D$ so that $h_{\lambda_0}$ has (DVC') (for e.g., $\lambda_0=0$). It remains to find
$\lambda_1$. Note, that 
\begin{equation*}
h_{\lambda_1}  = \frac1{\Phi-\lambda_1} = 
-\frac{1}{\lambda_1} - \frac{1}{\lambda_1^2(h-\lambda_1^{-1})}.
\end{equation*} 
Thus, if we can choose $\lambda_1 \in \widehat{\D}$ so that $1/\lambda_1 \notin h(\overline{\D})$
(which is possible for most of the choices of $r$ and $R$) we conclude that
$\Phi$ is of general position and has (DVC). By construction, $\Phi$ satisfies neither (MVC) nor (IAC).}
\end{example}

\begin{example}
\label{ex5}
{\rm Choosing the parameter $\beta$ in a good way we can modify the construction 
in Example \ref{ex4} so that 
$0< \tilde a<a< \tilde c< c< \tilde b<b<\tilde d <d$.
Then, repeating the construction, we build $\Phi$ of general position, 
which does not satisfy the condition (DVC') (see Figure \ref{fig:fig2}).} 
\end{example}

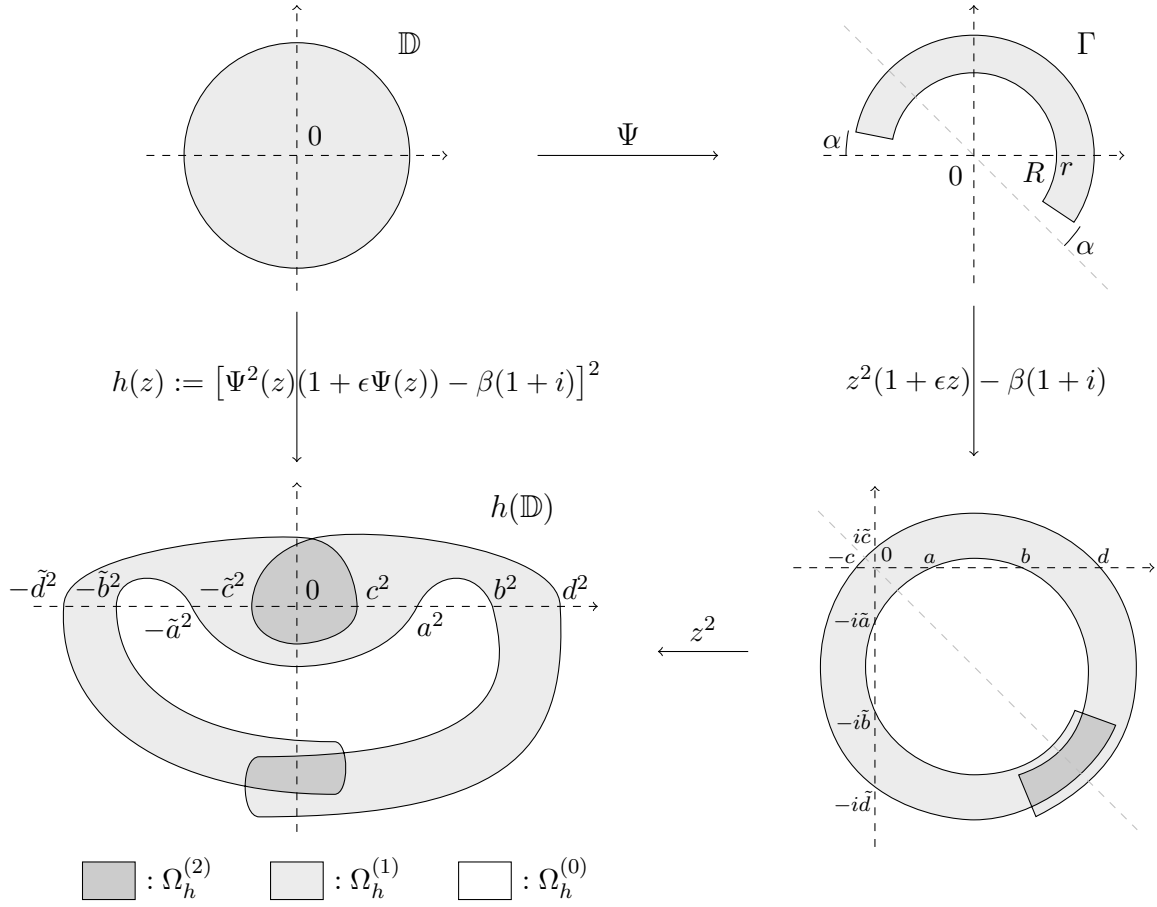
\begin{figure}
\centering
\begin{tikzpicture}

----------------
----------------

\fill[gray!15] (2,2) circle (1.5cm);

\draw[black] (2,2) circle (1.5cm);

\node[above right] at (2, 2) {{\small $0$}};

\node at (3.5, 3.5) {$\D$};

\draw[->, dashed] (2,0.2) -- (2,4);

\draw[->, dashed] (0,2) -- (4,2);

\fill[gray!15, shift={(1,0)}] ({10 + 1.6 * cos (-33.75)},{2 + 1.6 * sin (-33.75)}) arc(-33.75:168.75:1.6cm) -- ++(168.75:-0.5) arc(168.75:-33.75:1.1cm)      -- ++(-33.75:0.5)    -- cycle;

\draw[shift={(1,0)}] ({10 + 1.6 * cos (-33.75)},{2 + 1.6 * sin (-33.75)}) arc(-33.75:168.75:1.6cm) -- ++(168.75:-0.5) arc(168.75:-33.75:1.1cm)      -- ++(-33.75:0.5)    -- cycle;

\draw[->, dashed, shift={(1,0)}] (10,0.3) -- (10,4);

\draw[->, dashed, shift={(1,0)}] (8,2) -- (12,2);

\node at (12.5,3.5) {$\Gamma$};

\node[below right, shift={(1,0)}] at (11,2.05) {{\small $r$}} node[below right] at (11.5,2.05) {{\small $R$}} node[below left] at (11,2) {{\small $0$}};

\draw[shift={(1,0)}] ({10 + 1.7 * cos (168.75)},{2 + 1.7 * sin (168.75)}) arc (168.75:180:1.7cm);

\draw[shift={(1,0)}] ({10 + 1.7 * cos (-33.75)},{2 + 1.7 * sin (-33.75)}) arc (-33.75:-45:1.7cm);

\node[below left, shift={(1,0)}] at ({10.05 + 1.7 * cos (168.75)},{2.08 + 1.7 * sin (168.75)}) {{\small $\alpha$}};

\node[below right, shift={(1,0)}] at ({9.8 + 1.7 * cos (-33.75)},{1.95 + 1.7 * sin (-33.75)}) {{\small $\alpha$}};



\draw[->, shift={(1,0)}] (4.2,2) -- (6.6,2);

\draw[dashed, lightgray, shift={(1,0)}] ({9.3 + 1.45 * cos (135)},{2.7 + 1.45 * sin (135)}) -- ++(315:5);

\node[above, shift={(1,0)}] at (5.4,2) {$\Psi$};

----------------
----------------

\draw[->] (11,0) -- (11,-2);

\node[right, shift={(-1.85,0)}] at (11,-1) {{\small $z^2(1+\epsilon z)-\beta (1+i)$}};

\fill[gray!15, shift={(-1,-7.2)}, scale=1.2] ({10 + 1.8 * cos (-67.55)},{2 + 1.8 * sin (-67.55)}) ..controls +(1,0.5) and +(0,-0.5).. ({10 + 1.8 * cos (0)},{2 + 1.8 * sin (0)})
..controls +(0,1) and +(1,0).. ({10 + 1.7 * cos (90)},{2 + 1.7 * sin (90)})
..controls +(-1,0) and +(0,0.9).. ({10 + 1.7 * cos (180)},{2 + 1.7 * sin (180)})
..controls +(0,-1.5) and +(-0.35,0) .. ({10 + 1.7 * cos (270)},{2 + 1.7 * sin (270)})
..controls +(0.7,0) and +(-0.2,-0.5) .. ({10 + 1.7 * cos (337.5)},{2 + 1.7 * sin (337.5)})
-- ++(340:-0.48)
..controls +(-0.2,-0.5) and +(0.5,0) .. ({10 + 1.2 * cos (270)},{2 + 1.2 * sin (270)})
..controls +(-0.5,0) and +(0,-0.8) .. ({10 + 1.2 * cos (180)},{2 + 1.2 * sin (180)})
..controls +(0,0.8) and +(-0.5,0) .. ({10 + 1.2 * cos (90)},{2 + 1.2 * sin (90)})
..controls +(0.9,0) and +(0,0.5) .. ({10 + 1.27 * cos (-2)},{2 + 1.27 * sin (-2)})
..controls +(0,-0.5) and +(0.6,0.2) .. ({10 + 1.3 * cos (-67.55)},{2 + 1.3 * sin (-67.55)})
-- cycle
;

\fill[gray!40, shift={(-1,-7.2)}, scale=1.2] ({10 + 1.68 * cos (-67.55)},{2 + 1.68 * sin (-67.55)}) ..controls +(0.3,0.1) and +(-0.1,-0.33).. ({10 + 1.7 * cos (-22.45)},{2 + 1.7 * sin (-22.45)})
-- ++(340:-0.39)
..controls +(-0.1,-0.3) and +(0.36,0.1) .. ({10 + 1.3 * cos (-67.55)},{2 + 1.3 * sin (-67.55)})
-- cycle
;

\draw[shift={(-1,-7.2)}, scale=1.2] ({10 + 1.8 * cos (-67.55)},{2 + 1.8 * sin (-67.55)}) ..controls +(1,0.5) and +(0,-0.5).. ({10 + 1.8 * cos (0)},{2 + 1.8 * sin (0)})
..controls +(0,1) and +(1,0).. ({10 + 1.7 * cos (90)},{2 + 1.7 * sin (90)})
..controls +(-1,0) and +(0,0.9).. ({10 + 1.7 * cos (180)},{2 + 1.7 * sin (180)})
..controls +(0,-1.5) and +(-0.35,0) .. ({10 + 1.7 * cos (270)},{2 + 1.7 * sin (270)})
..controls +(0.7,0) and +(-0.2,-0.5) .. ({10 + 1.7 * cos (337.5)},{2 + 1.7 * sin (337.5)})
-- ++(340:-0.48)
..controls +(-0.2,-0.5) and +(0.5,0) .. ({10 + 1.2 * cos (270)},{2 + 1.2 * sin (270)})
..controls +(-0.5,0) and +(0,-0.8) .. ({10 + 1.2 * cos (180)},{2 + 1.2 * sin (180)})
..controls +(0,0.8) and +(-0.5,0) .. ({10 + 1.2 * cos (90)},{2 + 1.2 * sin (90)})
..controls +(0.9,0) and +(0,0.5) .. ({10 + 1.27 * cos (-2)},{2 + 1.27 * sin (-2)})
..controls +(0,-0.5) and +(0.6,0.2) .. ({10 + 1.3 * cos (-67.55)},{2 + 1.3 * sin (-67.55)})
-- cycle
;

\draw[->, dashed, shift={(-1,-7.2)}, scale=1.2] (8,{2 + 1.55 * sin (135)}) -- (12,{2 + 1.55 * sin (135)});

\draw[->, dashed, shift={(-1,-7.2)}, scale=1.2] ({10 + 1.55 * cos (135)},0) -- ({10 + 1.55 * cos (135)},4);

\draw[shift={(-1,-7.2)}, scale=1.2] node[above right] at ({9.95 + 1.55 * cos (135)}, {1.95 + 1.55 * sin (135)}) {{\tiny $0$}} node[above right] at (11.25, {2 + 1.45 * sin (135)}) {{\tiny $d$}} node[above right] at (10.4, {2 + 1.45 * sin (135)}) {{\tiny $b$}} node[above left] at (9.7, {2 + 1.45 * sin (135)}) {{\tiny $a$}} node[above left] at (8.8, {1.98 + 1.45 * sin (135)}) {{\tiny $-c$}} node[left] at ({10 + 1.45 * cos (135)},3.45) {{\tiny $i\tilde c$}} node[left] at ({10 + 1.45 * cos (135)},2.5) {{\tiny $-i\tilde a$}} node[left] at ({10 + 1.45 * cos (135)},1.4) {{\tiny $-i\tilde b$}} node[left] at ({10 + 1.45 * cos (135)},0.5) {{\tiny $-i\tilde d$}} [dashed, lightgray] ({9.3 + 1.45 * cos (135)},{2.7 + 1.45 * sin (135)}) -- ++(315:5);

-------------------------------
-------------------------------
	
--------------
--------------

\draw[->] (2,-0.08) -- (2,-2.08);

\node[right, shift={(-2.605,0)}] at (2,-1) {{\small $h(z):=\left[\Psi^2(z)(1+\epsilon \Psi(z))-\beta (1+i)\right]^2$}};
	
\draw[->] (8,-4.6) -- (6.8,-4.6) node[above] at (7.4,-4.6) {{\small $z^2$}};

	\fill[gray!15, shift={(-2.3,-1)}] (1.2,-3) ..controls +(0,0.8) and +(-0.5,0.1).. (4.5,-2.1)
	
	..controls +(0.5,-0.2) and +(0,0.2).. (5.1,-3)
	
	..controls +(-0,-0.4) and +(0.2,0).. (4.3,-3.5)
	
	..controls +(-0.2,0) and +(0,-0.4).. (3.7,-3)
	
	..controls +(0,0.4) and +(-0.7,-0.2).. (4.5,-2.1)
	
	..controls +(0.7,0.2) and +(0,0.8).. (7.8,-3)
	
	..controls +(0.1,-1.8) and +(3.5,0).. (3.8,-5.8)
	
	..controls +(-0.25,0) and +(-0.25,0).. (3.8,-5)
	
	..controls +(3.5,0) and +(0.2,-0.8).. (6.9,-3)
	
	..controls +(-0.2,0.5) and +(0.25,0.5).. (5.9,-3)
	
	..controls +(-0.35,-0.7) and +(0.35,0).. (4.3,-3.8)
	
	..controls +(-0.35,0) and +(0.35,-0.7).. (2.9,-3)
	
	..controls +(-0.25,0.5) and +(0,0.5).. (1.9,-3)
	
	..controls +(0,-1) and +(-2,0).. (4.8,-4.8)
	
	..controls +(0.2,0) and +(0.2,0).. (4.8,-5.5)
	
	..controls +(-2.2,0) and +(-0.1,-1.5).. (1.2,-3) -- cycle;

	\fill[gray!40, shift={(-2.3,-1)}] (4.5,-2.1) ..controls +(0.5,-0.1) and +(0,0.2).. (5.1,-3)
	
	..controls +(-0,-0.4) and +(0.2,0).. (4.3,-3.5)
	
	..controls +(-0.2,0) and +(0,-0.4).. (3.7,-3)
	
	..controls +(0,0.2) and +(-0.7,-0.2).. (4.5,-2.1) --cycle;

	\fill[gray!40, shift={(-2.3,-1)}] (3.8,-5) ..controls +(1,0.005) and +(-0.75,-0.05).. (4.925,-4.95)
	
	..controls +(0.03,-0.1) and +(0.2,0).. (4.8,-5.5)
	
	..controls +(-0.8,-0.005) and +(0.8,-0.1).. (3.61,-5.39)
	
	..controls +(0.03,+0.3) and +(-0.1,-0.01).. (3.8,-5);
	
	--------------
	--------------
	
	\node at (5,-2.7) {$h(\D)$};
	
	\draw[shift={(-2.23,-1.05)}] node[above right] at (4.2, -3) {{\small $0$}} node[above left] at (1.2, -3) {{\small $-\tilde d^2$}} node[above right] at (5, -3) {{\small $c^2$}} node[above left] at (3.7, -3) {{\small $-\tilde c^2$}} node[above right] at (7.61, -3) {{\small $d^2$}} node[above right] at (6.7, -3) {{\small $b^2$}} node[below right] at (5.68, -2.88) {{\small $a^2$}} node[below left] at (3, -2.9) {{\small $-\tilde a^2$}} node[above left] at (2.05, -3.005) {{\small $-\tilde b^2$}};
	
	\draw[->, dashed] (2,-7) -- (2,-2.35);
	
	
	\draw[->, dashed] (-1.5,-4) -- (6,-4);

	\draw[shift={(-2.3,-1)}] (1.2,-3) ..controls +(0,0.8) and +(-0.5,0.1).. (4.5,-2.1)
	
	 ..controls +(0.5,-0.1) and +(0,0.2).. (5.1,-3)
	 
	 ..controls +(-0,-0.4) and +(0.2,0).. (4.3,-3.5)
	 
	 ..controls +(-0.2,0) and +(0,-0.4).. (3.7,-3)
	 
	 ..controls +(0,0.2) and +(-0.7,-0.2).. (4.5,-2.1)
	 
	 ..controls +(0.7,0.2) and +(0,0.8).. (7.8,-3)
	 
	 ..controls +(0.1,-1.8) and +(3.5,0).. (3.8,-5.8)
	 
	 ..controls +(-0.25,0) and +(-0.25,0).. (3.8,-5)
	 
	 ..controls +(3.5,0) and +(0.2,-0.8).. (6.9,-3)
	 
	 ..controls +(-0.2,0.5) and +(0.25,0.5).. (5.9,-3)
	 
	 ..controls +(-0.35,-0.7) and +(0.35,0).. (4.3,-3.8)
	 
	 ..controls +(-0.35,0) and +(0.35,-0.7).. (2.9,-3)
	 
	 ..controls +(-0.25,0.5) and +(0,0.5).. (1.9,-3)
	 
	 ..controls +(0,-1) and +(-2,0).. (4.8,-4.8)
	 
	 ..controls +(0.2,0) and +(0.2,0).. (4.8,-5.5)
	 
	 ..controls +(-2.2,0) and +(-0.1,-1.5).. (1.2,-3);
	
	\draw [fill=gray!40, shift={(0.65,0.6)}] (-1.5,-8.5) rectangle ++(0.7,0.5) node[right] at (-0.8,-8.21) {{\small $: \Omega _h^{(2)}$}};
	\draw [fill=gray!15, shift={(0.65,0.6)}] (1,-8.5) rectangle ++(0.7,0.5) node[right] at (1.7,-8.21) {{\small $: \Omega _h^{(1)}$}};
	\draw [shift={(0.65,0.6)}] (3.5,-8.5) rectangle ++(0.7,0.5) node[right] at (4.2,-8.21) {{\small $: \Omega _h^{(0)}$}};

	\end{tikzpicture}
	\caption{Example of $h$ satisfying (DVC')} \label{fig:fig1}
\end{figure}

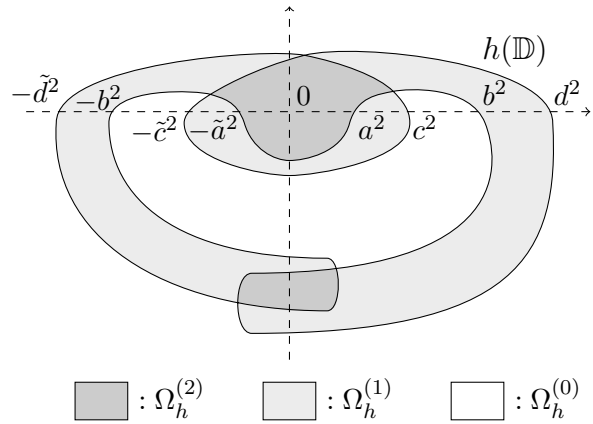
\begin{figure}
	\centering
	\begin{tikzpicture}
 
 -------------------------------
 -------------------------------
 
 --------------
 --------------
 
 \fill[gray!15, shift={(-2.3,-0.65)}] (1.2,-3) (1.2,-3) ..controls +(0,0.8) and +(-0.5,0.1).. (4.5,-2.1)
 
 ..controls +(0.5,-0.1) and +(-0.2,0.2).. (5.68,-2.58)
 
 ..controls +(0.2,-0.2) and +(0,0.1).. (5.9,-3)
 
 ..controls +(0,-0.5) and +(0.35,0).. (4.3,-3.7)
 
 ..controls +(-0.35,0) and +(0,-0.5).. (2.9,-3)
 
 ..controls +(0,0.17) and +(-0.08,-0.05).. (3.25,-2.62)
 
 ..controls +(0.08,0.05) and +(-0.7,-0.18).. (4.5,-2.1)
 
 ..controls +(0.7,0.18) and +(0,0.8).. (7.8,-3)
 
 ..controls +(0.1,-1.8) and +(3.5,0).. (3.8,-5.8)
 
 ..controls +(-0.25,0) and +(-0.25,0).. (3.8,-5)
 
 ..controls +(3.5,0) and +(0.2,-0.5).. (6.9,-3)
 
 
 ..controls +(-0.2,0.5) and +(0.2,0.04).. (5.68,-2.58)
 
 ..controls +(-0.2,-0.04) and +(0.1,0.35).. (5.1,-3)
 
 ..controls +(-0.1,-0.4) and +(0.2,0).. (4.3,-3.5)
 
 ..controls +(-0.2,0) and +(0.1,-0.35).. (3.7,-3)
 
 ..controls +(-0.1,0.3) and +(0.15,-0.05).. (3.25,-2.62)
 
 
 ..controls +(-0.15,0.05) and +(0,0.5).. (1.9,-3)
 
 ..controls +(0,-1) and +(-2,0).. (4.8,-4.8)
 
 ..controls +(0.2,0) and +(0.2,0).. (4.8,-5.5)
 
 ..controls +(-2.2,0) and +(-0.1,-1.5).. (1.2,-3) -- cycle;

 \fill[gray!40, shift={(-2.3,-0.65)}] (4.5,-2.1) ..controls +(0.5,-0.1) and +(-0.2,0.2).. (5.68,-2.58)
 
 ..controls +(-0.2,-0.04) and +(0.1,0.35).. (5.1,-3)
 
 ..controls +(-0.1,-0.4) and +(0.2,0).. (4.3,-3.5)
 
 ..controls +(-0.2,0) and +(0.1,-0.35).. (3.7,-3)
 
 ..controls +(-0.1,0.3) and +(0.15,-0.05).. (3.25,-2.62) 
 
 ..controls +(0.08,0.05) and +(-0.7,-0.18).. (4.5,-2.1) --cycle;

 \fill[gray!40, shift={(-2.3,-0.65)}] (3.8,-5) ..controls +(1,0.005) and +(-0.75,-0.05).. (4.925,-4.95)
 
 ..controls +(0.03,-0.1) and +(0.2,0).. (4.8,-5.5)
 
 ..controls +(-0.8,-0.005) and +(0.8,-0.1).. (3.61,-5.39)
 
 ..controls +(0.03,+0.3) and +(-0.1,-0.01).. (3.8,-5);
 
 --------------
 --------------
 
 \node at (5,-2.7) {$h(\D)$};
 
 \draw[shift={(-2.23,-0.65)}] node[above right] at (4.18, -2.9) {{\small $0$}} node[above left] at (1.33, -2.9) {{\small $-\tilde d^2$}} node[below right] at (5, -2.75) {{\small $a^2$}} node[below left] at (3.7, -2.75) {{\small $-\tilde a^2$}} node[above right] at (7.6, -2.9) {{\small $d^2$}} node[above right] at (6.65, -2.9) {{\small $b^2$}} node[below right] at (5.72, -2.75) {{\small $c^2$}} node[below left] at (2.9, -2.77) {{\small $-\tilde c^2$}} node[above left] at (2.15, -3) {{\small $-b^2$}};
 
 \draw[->, dashed] (2,-6.8) -- (2,-2.1);
 
 
 \draw[->, dashed] (-1.5,-3.5) -- (6,-3.5);

 \draw[shift={(-2.3,-0.65)}] (1.2,-3) ..controls +(0,0.8) and +(-0.5,0.1).. (4.5,-2.1)
 
 ..controls +(0.5,-0.1) and +(-0.2,0.2).. (5.68,-2.58)
 
 ..controls +(0.2,-0.2) and +(0,0.1).. (5.9,-3)
 
 ..controls +(0,-0.5) and +(0.35,0).. (4.3,-3.7)
 
 ..controls +(-0.35,0) and +(0,-0.5).. (2.9,-3)
 
 ..controls +(0,0.17) and +(-0.08,-0.05).. (3.25,-2.62)
 
 ..controls +(0.08,0.05) and +(-0.7,-0.18).. (4.5,-2.1)
 
 ..controls +(0.7,0.18) and +(0,0.8).. (7.8,-3)
 
 ..controls +(0.1,-1.8) and +(3.5,0).. (3.8,-5.8)
 
 ..controls +(-0.25,0) and +(-0.25,0).. (3.8,-5)
 
 ..controls +(3.5,0) and +(0.2,-0.5).. (6.9,-3)
 
 
 ..controls +(-0.2,0.5) and +(0.2,0.04).. (5.68,-2.58)
 
 ..controls +(-0.2,-0.04) and +(0.1,0.35).. (5.1,-3)
 
 ..controls +(-0.1,-0.4) and +(0.2,0).. (4.3,-3.5)
 
 ..controls +(-0.2,0) and +(0.1,-0.35).. (3.7,-3)
 
 ..controls +(-0.1,0.3) and +(0.15,-0.05).. (3.25,-2.62)
 
 
 ..controls +(-0.15,0.05) and +(0,0.5).. (1.9,-3)
 
 ..controls +(0,-1) and +(-2,0).. (4.8,-4.8)
 
 ..controls +(0.2,0) and +(0.2,0).. (4.8,-5.5)
 
 ..controls +(-2.2,0) and +(-0.1,-1.5).. (1.2,-3);
 
 \draw [fill=gray!40, shift={(0.65,0.9)}] (-1.5,-8.5) rectangle ++(0.7,0.5) node[right] at (-0.8,-8.21) {{\small $: \Omega _h^{(2)}$}};
 \draw [fill=gray!15, shift={(0.65,0.9)}] (1,-8.5) rectangle ++(0.7,0.5) node[right] at (1.7,-8.21) {{\small $: \Omega _h^{(1)}$}};
 \draw [shift={(0.65,0.9)}] (3.5,-8.5) rectangle ++(0.7,0.5) node[right] at (4.2,-8.21) {{\small $: \Omega _h^{(0)}$}};
 
 	\end{tikzpicture}
 \caption{Example of $h$ which does not satisfy (DVC')} \label{fig:fig2}
\end{figure}


\section{Necessary conditions}
\label{nec}

In this section we prove Proposition \ref{ness}. The proof goes along the same 
lines as the proofs in \cite{barl} and we include it for the sake of completeness only. 
Assume that $\Phi$ is not $N$-valent and so, for some
$\mu$, the equation  $\Phi(z) = \mu$ has $N+1$ solutions $z_1, z_2, \dots z_{N+1}$ in $\D$ counting 
multiplicities. We assume that these solutions are distinct. If they are not distinct the argument requires an obvious
modification (see \cite{barl} for details). We will show that the adjoint operator 
$T_{\bar \Phi}$ has an eigenvector and so $T_\Phi$ is not hypercyclic, a contradiction. 

We construct this eigenvector of  $T_{\bar \Phi}$ as a linear combination of Cauchy kernels 
$k_{z_m}$, where $k_\lambda(z) = \frac{1}{1-\bar \lambda z}$. Recall that for any 
antianalytic Toeplitz operator we have $T_{\overline{\phi}} k_\lambda = 
\overline{\phi(\lambda)} k_\lambda$. 

Put $f = \sum_{m=1}^{N+1}\beta_m k_{z_m}$,
where $\beta_m$ are some complex coefficients. Note that
$T_{\overline{\Phi}} = T_{R^*(z) +\overline{\phi(z)}}$, where
$$
R^*(z) = \sum_{k=0}^{N_1} \bar c_k z^k + \sum _{l=1}^r 
\sum_{j=1}^{k_l}\frac{\bar \alpha_{l, j}}{(z-\bar \eta_l)^{j}}
$$
is analytic and bounded in $\D$. Hence, using the fact that $R(1/z_m) + \phi(z_m) = \mu$, we get
\begin{equation}
\label{bab2}
\begin{aligned}
T_{\overline{\Phi}} f(z) = \sum_{m=1}^{N+1} \beta_m\bigg(
\frac{R^*(z)}{1-\overline{z_m} z} & +\frac{\overline{\phi(z_m)}}{1-\overline{z_m} z}\bigg) \\
& = \bar \mu f(z) +\sum_{m=1}^{N+1}\beta_m 
\frac{R^*(z) - \overline{R(1/z_m)}}{1-\overline{z_m} z}.
\end{aligned}
\end{equation}
Clearly, 
$$
\frac{R^*(z) - \overline{R(1/z_m)}}{1-\overline{z_m} z}
\in {\rm span} \bigg\{z^k, \frac{1}{(z-\bar \eta_l)^{j}}: \ 0\le k\le N_1-1, \ 1\le l\le r, 
\ 1\le j\le k_l \bigg\}. 
$$
Since the dimension of this span is $N$, we can find nontrivial $\beta_m$ such that the last
sum in \eqref{bab2} is identically zero, and so $T_{\overline{\Phi}} f = \bar \mu f$. 
\qed
\smallskip

\section{Sufficient conditions}\label{suff}

The proofs of Theorems \ref{main2}--\ref{main4} are each based on 
Proposition \ref{red}, which in turns is based on an application of the Godefroy--Shapiro 
Criterion (see \cite{gosh}, \cite[Corollary 1.10]{bm} or \cite[Theorem 3.1]{gp}).

\begin{theorem*}[Godefroy--Shapiro Criterion] 
If, for a bounded linear operator $T$ on a separable Banach 
space $X$, both $\cup_{|\lambda|<1}{\rm Ker}\,(T-\lambda I)$ 
and $\cup_{|\lambda|>1}{\rm Ker}\,(T-\lambda I)$  span a dense subspace in $X$, then $T$ is hypercyclic.
\end{theorem*}

We will then first search for the eigenvectors of $T_{\Phi}$.

\subsection{Computation of eigenfunctions}
Let us denote by $S_n(g)$ the $n$th partial sum of the Taylor series of $g\in H^2$ at $0$. 
Put $Q(z) = \prod_{l=1}^{r}(1-\eta_l z)^{k_l}$.

We recall that 
$T_{\bar{z}}$ is the backward shift $S^*$ on $H^2$, i.e.,
\[
T_{\bar{z}}f=S^*f=\frac{f(z)-f(0)}{z}\quad \text{and}\quad T_{\bar{z}^k}f=(S^*)^kf=\frac{f(z)-S_{k-1}(f)}{z^k}.
\]
Then, for $\eta \in \Dh$, $k\geq 1$ and $g\in H^2$,
$$
\begin{aligned}
(S^*-\eta I)^kg & = \sum_{j=0}^k\binom{k}{j}(-\eta)^{k-j}(S^*)^jg \\
& =  g\sum_{j=0}^k\frac{\binom{k}{j}(-\eta)^{k-j}}{z^j}-\sum_{j=1}^k\binom{k}{j}(-\eta)^{k-j}\frac{S_{j-1}(g)}{z^j}\\
& =  \frac{(1-\eta z)^k}{z^k}g-\frac{1}{z^k}\sum_{j=1}^k\binom{k}{j}(-\eta z)^{k-j}S_{j-1}(g).
\end{aligned}
$$
Thus, for $f\in H^2$, applying the previous to $g=[(S^*-\eta I)^k]^{-1}f$, we obtain
\[
[(S^*-\eta I)^k]^{-1}f=\frac{z^kf+p_k}{(1-\eta z)^k},
\]
where $p_k$ is some polynomial of degree at most $k-1$. It follows that
$$
\begin{aligned}
T_{\Phi}f & = \sum_{k=0}^{N_1}c_k (S^*)^k f+\sum _{l=1}^r \sum_{j=1}^{k_l}
\alpha_{l, j}  [(S^*-\eta_l)^{j}]^{-1} f +\varphi f \\
& =  \sum_{k=0}^{N_1} \frac{c_kf +q_k}{z^k}+\sum _{l=1}^r \sum_{j=1}^{k_l}
\alpha_{l, j} 
\frac{z^{j}f+p_{l, j}}{(1-\eta_l z)^{j}}+\varphi f \\
& = \Phi f +\sum_{k=1}^{N_1} \frac{q_k}{z^k}+\sum _{l=1}^r \sum_{j=1}^{k_l}
\frac{ \alpha_{l, j} p_{l, j}}{(1-\eta_l z)^{k_l}},
\end{aligned}
$$
where  $p_{l, j}$ is a polynomial of degree at most $k_l-1$ 
and $q_k$ is a polynomial of degree at most $k-1$ (while $q_0=0$).

Therefore, the equation $T_{\Phi}f_\lambda = \lambda f_\lambda$ is equivalent, 
for some polynomials $p$ and $q$ with degree at most $N_2-1$ and $N_1 -1$ respectively, to
\[
f_\lambda =\frac{z^{N_1} p + Q q}{z^{N_1} Q (\Phi-\lambda)}.
\]
Clearly, $f_\lambda \in H^2$ for any $\lambda \in \C\setminus \overline{\Phi(\D)}$, and so $f_\lambda$
is an eigenvector of $T_{\Phi}$. 
Note that this holds true {\it for any choice} of the polynomials $p$ and $q$ with degree at most $N_2-1$ 
and $N_1-1$ respectively.
\medskip

\subsection{Proof of Proposition \ref{red}}
By \eqref{spe}, we can find two open sets $U_0\subset \D\cap (\C\setminus \overline{\Phi(\D)})$ 
and $U_1\subset \widehat{\D} \cap (\C\setminus \overline{\Phi(\D)})$ 
consisting of eigenvalues of $T_{\Phi}$. Let $\lambda_0\in U_0$. 

Expanding $f_{\lambda}$ around $\lambda_0$, for $\lambda$ close to $\lambda_0$, we get
\[
f_{\lambda}=\frac{z^{N_1} p + Qq}{Q z^{N_1} (\Phi - \lambda_0)} \sum _{k\geq 0} h^k(\lambda-\lambda_0)^k,
\]
with $h=\frac{1}{\Phi-\lambda_0}$. Thus, if $f \in H^2$ is orthogonal to all such $f_{\lambda}$, then
\[
f\perp \frac{z^{N_1} p + Qq}{Q z^{N_1} (\Phi - \lambda_0)} h^k,\qquad k\geq 0
\]
for any polynomials $p$ and $q$ with degree at most $N_2-1$ and $N_1-1$ respectively.
Note that any polynomial of degree at most $N-1$ can be represented as $z^{N_1} p + Qq$
for some $p$ and $q$, whence
\[
f\perp (z^{N_1} Q (\Phi - \lambda_0))^{-1} z^j h^k,\qquad 0\le j\le N-1, \ \ k\geq 0.
\]
Now, by assumption, the family $\{1,z,\ldots,z^{N-1}\}$ is a cyclic set for the operator 
$T_h$, i.e., the family $\{z^j h^k:\,0\le j\le N-1, k\geq 0\}$ is complete in $H^2$. 
Since $z^{N_1} Q(\Phi - \lambda_0)$ is a nonvanishing function in $A(\D)$, 
it is invertible in $A(\D)$, and so the family 
\[\{ (z^{N_1} Q(\Phi - \lambda_0))^{-1} z^j h^k:\,0\le j\le N-1, k\geq 0\}\]
is also complete. Thus, $f=0$. 

Since the same argument works for $U_1$, we see that the families 
$\{f_\lambda\}_{\lambda\in U_0}$ and $\{f_\lambda\}_{\lambda\in U_1}$ are complete in $H^2$. Hence, 
by the Godefroy--Shapiro Criterion, $T_\Phi$ is hypercyclic. \qed
\medskip
\begin{remark}
\label{yuu}
{\rm Note that for the conclusion of Proposition \ref{red} to hold
it is sufficient to assume that there exist
$\lambda_0 \in \D\cap (\C\setminus \overline{\Phi(\D)})$ 
and $\lambda_1 \in \widehat{\D} \cap (\C\setminus \overline{\Phi(\D)})$
such that for $h_0=\frac{1}{\Phi-\lambda_0}$  and
$h_1=\frac{1}{\Phi-\lambda_1}$ the family 
$\{1,z,\ldots,z^{N-1}\}$ is cyclic for $T_{h_0}$ and for $T_{h_1}$. }
\end{remark}
\medskip

Before giving the proofs of Theorems \ref{main2}--\ref{main4}, let us recall Solomyak's 
results that we shall apply.


\subsection{Solomyak's theorems}
Even if we use only a very special case of the results from \cite{sol}, we find it appropriate to 
give a short survey of them. Recall that a matrix-valued function of size $n\times p$ whose entries 
are in $H^2$ is said to be {\it outer} if $p\ge n$ and the greatest common inner divisor of its minors of order 
$n$ is 1. In this definition it is possible that $p=\infty$.

\begin{theorem}[B. Solomyak, 1987]
\label{so1}
Let $h \in A(\D)$ and let the set $h(\T)$ be a finite union of $\mathcal{C}^2$-smooth Jordan arcs. 
If $h$ satisfies the property {\rm (IAC)}, then the set 
$\{u_1,\ldots,u_m \} \subset H^2$ is cyclic for the operator $T_h$ if and only 
if the following two conditions hold:
\medskip

{\bf (P1)} For any $\zeta_1, \zeta_2, \ldots, \zeta_l \in \D$ such that 
$h(\zeta_1) = \ldots = h(\zeta_l) = a$, and the value $a$ is taken at the point $\zeta_j$ 
with multiplicity $K_j$, we have
$$
rank\ [u_i(\zeta_j), u_i'(\zeta_j), 
\ldots u_i^{(K_j - 1)} (\zeta_j)]_{\substack{1\leq i\leq m\\ 1\leq j \leq l}} 
= \sum\limits_{j=1}^l K_j;
$$

{\bf (P2)} For any $z \in \T$ and for sufficiently small neighbourhood $V_z$ such that in 
$V_z \cap \D$ there exist $k$ one-to-one branches $\psi_1 =id$, $\psi_2, \ldots, \psi_k$ 
of the function $h^{-1} \circ h$, the matrix-function 
$[(u_i \circ \psi_j) (\zeta)]_{1\le i \le m, \ 1\le j \le k}$ is outer in $V_z \cap \D$.
\end{theorem}

Note that this theorem applies also to the case when $\{u_i\}_i$ is an infinite sequence, i.e., $m=\infty$. 
                 
\begin{theorem}[B. Solomyak, 1987]
\label{so2}
Let $h$ be analytic in $\overline{\D}$. If $h$ is of general position and satisfies the property {\rm (DVC')}, 
then the set $\{u_1,\ldots,u_m \}$, $m \in \mathbb{N}$, is cyclic for the operator $T_h$ 
if and only if the conditions {\rm (P1), (P2)} and {\rm (P3)} hold, where
\medskip

{\bf (P3)} For any hole $G \subset \Omega_h^{(0)}$ there exist $i, j \leq m$ such that 
$(u_i \circ \nu^{-1}) / (u_j \circ \nu^{-1}) |_{\partial G} \notin N(G)$, 
where $\nu = h|_{\T}$ and $N(G)$ is the Nevanlinna class in $G$.
\end{theorem}

We shall apply Solomyak's theorems only to the set
$\{1,z,\ldots,z^{N-1}\}$ which, by \cite[Sect. 1.3, Rem. 3]{sol}, satisfies
the conditions (P1), (P2) and (P3) for any $N$-valent $h$ satisfying the conditions
of Solomyak's theorems \ref{so1} and  \ref{so2}.
\medskip


\subsection{Proof of Theorems \ref{main2}--\ref{main4}}

\begin{proof}[Proof of Theorem \ref{main2}]   
If $\Phi$ satisfies (MVC), then also for any $\lambda \in \C\setminus \overline{\Phi(\D)}$
the function $h = \frac{1}{\Phi - \lambda}$ has the property that for any
$w\in h(\overline{\D})$ the equation $h(z) = w$ has exactly $N$ solutions in 
$\overline{\D}$. By \cite[Proposition 3.1]{barl} the family
$\{z^j h^k:\,0\le j \le N-1, k\ge 0\}$ is complete in $H^2$. Thus, by Proposition \ref{red},
$T_\Phi$ is hypercyclic. 
\end{proof}
\medskip
\begin{proof}[Proof of Theorem \ref{main3}]   
Since $\Phi$ satisfies (IAC), there is $\lambda \in \C\setminus \overline{\Phi(\D)}$ 
such that Theorem \ref{so1} applies to the function $h = \frac{1}{\Phi - \lambda}$. Thus, the family 
$\{z^j h^k:\,0\le j\le N-1, k\geq 0\}$ is complete in $H^2$. Let us show that for any other
$\mu \in \C\setminus \overline{\Phi(\D)}$ and $h_{\mu} = \frac{1}{\Phi - \mu}$
the family $\{z^j h_{\mu}^k:\, 0\le j\le N-1, k\geq 0\}$ is also complete in $H^2$. 

Note that $h_1 = c_1 + \frac{c_2}{h-c_3}$, where $c_1, c_2, c_3$ are some constants 
(which can be written explicitly in terms of $\lambda$ and $\mu$, see \eqref{hhh}) 
and $c_3 \notin h(\overline{\D})$. 
By condition (IAC), the closed domain $\Omega = h(\overline{\D})$ has no holes, and so
the function $z$ can be approximated by functions in ${\rm span} \{(z-c_3)^{-k}:\, k\ge 0\}$ 
uniformly in $\Omega$. Thus, $h\in \overline{{\rm span}} \{(h-c_3)^{-k}:\, k\ge 0\}$ in $H^2$, 
whence
$$
\begin{aligned}
\overline{{\rm span}}
\{z^j h_1^k:\,  0\le j\le & N-1,  k\geq 0\} \\ 
& = \overline{{\rm span}} \{z^j h^k:\, 0\le j\le N-1, k\geq 0\} = H^2.
\end{aligned}
$$
Now, $T_\Phi$ is hypercyclic by Proposition \ref{red}. 
\end{proof}
\medskip
\begin{proof}[Proof of Theorem \ref{main4}]   
Let $\lambda_0 \in \D\cap (\C\setminus \overline{\Phi(\D)})$ and
$\lambda_1 \in \widehat{\D} \cap (\C\setminus \overline{\Phi(\D)})$. 
By Remark \ref{DVCprime}, both functons $h_0 = \frac{1}{\Phi - \lambda_0}$ and 
$h_1 = \frac{1}{\Phi - \lambda_1}$ belong to $A(\D)$, are of general position 
and satisfy (DVC'). Then, by Theorem \ref{so2}, the families $\{z^j h_0^k: 0\le j\le N-1, k\geq 0\}$
and  $\{z^j h_1^k: 0\le j\le N-1, k\geq 0\}$ are complete in $H^2$. 
Therefore, by Proposition \ref{red} and Remark \ref{yuu}, $T_\Phi$ is hypercyclic. 
\end{proof}

\end{document}